\documentclass{amsart}

\usepackage{amsfonts,amssymb,amsmath,mathtools}
\usepackage{ stmaryrd }
\usepackage{hyperref}
\usepackage[nobysame]{amsrefs}

\usepackage{enumitem}

\theoremstyle{plain}
\newtheorem*{Theorem}{Theorem}
\newtheorem{theorem}{Theorem}[section]
\newtheorem{proposition}[theorem]{Proposition}
\newtheorem{lemma}[theorem]{Lemma}
\newtheorem{corollary}[theorem]{Corollary}

\theoremstyle{definition}
\newtheorem{definition}[theorem]{Definition}
\newtheorem{example}[theorem]{Example}

\theoremstyle{remark}
\newtheorem{remark}[theorem]{Remark}

\numberwithin{equation}{theorem}
\DeclareMathOperator{\Supp}{Supp }

\DeclareMathOperator{\spec}{Spec }
\DeclareMathOperator{\id}{inj-dim}
\DeclareMathOperator{\fd}{flat-dim}
\DeclareMathOperator{\pd}{proj-dim}
\DeclareMathOperator{\Thick}{Thick}

\DeclareMathOperator{\lowlen}{\ell \ell}
\DeclareMathOperator{\LF}{\mathsf{L}F^\varphi}
\DeclareMathOperator{\DDfg}{\mathsf{D}^{fg}}
\DeclareMathOperator{\DDfl}{\mathsf{D}^{fl}}

\DeclareMathOperator{\IIfg}{\mathsf{I}^{fg}}
\DeclareMathOperator{\IIfl}{\mathsf{I}^{fl}}
\DeclareMathOperator{\DDfgh}{\mathsf{D}_{+}^{fg}}

\DeclareMathOperator{\DDfgl}{\mathsf{D}_{-}^{fg}}

\DeclareMathOperator{\DDfgb}{\mathsf{D}_{b}^{fg}}
\DeclareMathOperator{\DDflb}{\mathsf{D}_{b}^{fl}}
\DeclareMathOperator{\CC}{\mathsf{S}}
\DeclareMathOperator{\CCb}{\mathsf{S}_{b}}
\DeclareMathOperator{\CCh}{\mathsf{S}_{+}}
\DeclareMathOperator{\CCl}{\mathsf{S}_{-}}
\DeclareMathOperator{\CCfg}{\mathsf{S}^{fg}}
\DeclareMathOperator{\CCfl}{\mathsf{S}^{fl}}
\DeclareMathOperator{\DDh}{\mathsf{D}_{+}}
\DeclareMathOperator{\DDl}{\mathsf{D}_{-}}
\DeclareMathOperator{\DDb}{\mathsf{D}_{b}}
\DeclareMathOperator{\ann}{ann}
\DeclareMathOperator{\phiR}{R^\varphi}
\DeclareMathOperator{\phiRi}{R^{\varphi^{i}} }
\newcommand{\m}{\mathfrak{m}}

\newcommand{\p}{\mathfrak{p}}
\newcommand{\n}{\mathfrak{n}}
\newcommand{\qis}{\simeq}
\newcommand{\tensor}{\otimes}
\newcommand{\dtensor}{\otimes^\mathbf{L}}
\def\Hom{\operatorname{Hom}}
\DeclareMathOperator{\RHom}{\operatorname{RHom}}

\def\NN{\mathbb N}

\def\DD{\mathsf{D}}
\def\II{\mathsf{I}}
\def\blank{\underline{\hspace{8pt}}}
\newcommand{\surject}{\twoheadrightarrow}

\begin{document}

\title[Contracting Endomorphisms]{Base Change Along the Frobenius Endomorphism And The Gorenstein Property}

\author{Pinches Dirnfeld}
\address{Pinches Dirnfeld, Department Of Mathmatics, University of Utah, 155 South 1400 East, JWB 233, Salt Lake City, UT 84112}
\email{dirnfeld@math.utah.edu}

\date{\today}


\thanks{I thank my advisor Srikanth Iyengar for the many helpful discussions and reading many versions of this paper. This work was partly supported
by a grant from the National Science Foundation, DMS-1700985.}

\begin{abstract}
Let $R$ be a local ring of positive characteristic and $X$ a complex with nonzero finitely generated homology and finite injective dimension. We prove that if derived base change of $X$ via the Frobenius (or more generally, via a contracting) endomorphism has finite injective dimension then $R$ is Gorenstein. 
\end{abstract}

\maketitle

\section{Introduction}

Kunz \cite{MR252389} proved that a local ring $(R, \m ,k)$ of positive characteristic is regular if and only if some (equivalently, every) power of the Frobenius endomorphism is flat as an $R$-module. Since then analogous characterizations of other properties of the ring, such as complete intersections (by Rodicio \cite{MR963004}), Gorenstein (by Goto \cite{MR447212}) and Cohen-Macualay (by Takahashi and Yoshino \cite{MR2073291}), have been obtained. Many of these results have been established for the larger family of \emph{contracting} endomorphisms. Following \cite{MR2197067}, an endomorphism $\varphi \colon R \to R$ is said to be \emph{contracting} if $\varphi^i(\m) \subseteq \m^2$ for some $i > 0$. The Frobenius endomorphism is one example but there are many interesting examples even when $R$ is of characteristic 0. Avramov, Iyengar and Miller \cite{MR2197067} generalized Kunz's theorem to apply to any contracting endomorphism. For other results concerning contracting endomorphisms see, for example, Avramov, Hochster, Iyengar and Yao \cite{MR2915536}, Rahmati \cite{Rah09} and Nasseh and Sather-Wagstaff \cite{NaSa15}. 

In this paper we study homological properties of modules and complexes under base change along contracting endomorphisms. Given an endomorpishm $\varphi \colon R \to R$, we write $\phiR$ for the $R$ bimodule with the right module structure induced by $\varphi$ and the left usual $R$ module structure. Thus given an $R$-complex $X$ the base change along $\varphi$ is $\phiR \dtensor_RX$ where $R$ acts on the left through $\phiR$. The main result of this work is the following, proved in Section \ref{mainsection}:
\begin{theorem}\label{main result}
	Let $\varphi \colon R \to R$ be a contracting endomorphism. The following conditions are equivalent.
	\begin{enumerate}
		\item $R$ is Gorenstein.
		\item There exists an $R$-complex $X$ with nonzero finitely generated homology and finite injective dimension for which the base change $\phiR \dtensor_RX$ has finite injective dimension.
		\item For every $X$ with nonzero finitely generated homology and finite injective dimension the base change $\phiR \dtensor_RX$ has finite injective dimension.
	\end{enumerate}
\end{theorem}
An $R$-complex has \emph{finite injective dimension} if it is quasi-isomorphic to a bounded complex of injective modules. 

In the theorem above (i)$\Rightarrow$(iii) holds because in a Gorenstein local ring complexes of finite injective dimension coincide with complexes of finite projective dimension. For (iii)$\Rightarrow$(ii) we only need to show that every local ring has a complexes of finite injective dimension with nonzero finitely generated homology. (ii)$\Rightarrow$(i) is the crucial implication. This is proven in two steps. 1) When $H(X)$ is finitely generated, if $\phiRi \dtensor_RX$ is bounded in homology for $i \gg 0$ then $X$ has finite projective dimension; this follows from well known arguments, see \ref{finite pd} for details. 2) When $X$ is a complex with nonzero finite length homology and finite injective dimension, if the base change $\phiR \dtensor_RX$ has finite injective dimension then  the same holds for $\phiRi \dtensor_RX$ for every $i>0$, see \ref{inj-DLH}. The key tool in the proof of the second step is a theorem of Hopkins \cite{MR932260} and Neeman \cite{MR1174255} concerning perfect complexes.

Theorem \ref{main result} is analogous statement to the following characterization of Gorenstein rings by Falahola and Marley \cite{FM} 
\begin{theorem}
 Let $\varphi \colon R \to R$ be a contracting endomorphism where $R$ is a Cohen-Macaulay local ring and $\omega_R$ a canonical module. Then $R$ is Gorenstein if and only if  $\phiR \tensor_R \omega_R$ has finite injective dimension.
\end{theorem}
 In \cite{FM} the authors ask: when $R$ is a local ring with a dualizing complex $D$, if $\phiR \dtensor_R D$ has finite injective dimension is then $R$ necessarily  Gorenstein? Theorem \ref{main result} gives an affirmative answer.

\section{Homological Invariants}
In this Section we recall basic definitions and results for use in Section \ref{mainsection}. Throughout this paper $R$ will be a commutative Noetherian ring. We write $\DD(R)$ for the derived category of $R$-complexes, with the convention that complexes are graded below i.e. we write $$X = \dots \to X_n \to X_{n-1} \to \dots$$ We write $X\qis Y$ when $X$ is isomorphic to $Y$ in $\DD(R)$.

\begin{definition}
Given an $R$-complex $X$ we set
$$ \sup H(X) = \sup \{ i\mid H_i(X) \neq 0 \} \text{ and } \inf H(X) = \inf \{ i\mid H_i(X) \neq 0 \}$$
Thus $\sup (0) = -\infty $ and $\inf (0) = \infty $.
A complex $X$ is said to be \emph{homologically bounded above} if $\sup H(X)<\infty$. Similarly, $X$  is  \emph{homologically bounded below} if $\inf H(X) > -\infty$ and $X$ is \emph{homologically bounded} if it is homologically bounded above and below.\end{definition}
Let $\CC$ be a full subcategory of $\DD(R)$. We make the following conventions:
\begin{enumerate}
 \item$\CCl \coloneqq \{ X \in \CC \mid \sup H(X)<\infty \}$
\item $\CCh\coloneqq \{ X \in \CC \mid \inf H(X)> -\infty \}$
\item $\CCb \coloneqq \CCh \cap \CCl$
\item The subcategory of of complexes in $\CC$ with degree-wise finitely generated (resp. finite length) homology is denoted $\CCfg$ (resp. $\CCfl$)
\end{enumerate} 

In $\DD(R)$ we have derived functors $\RHom_R(\blank,\blank)$ and $\blank \dtensor_R \blank$. For a detailed description on how these functors are defined, we refer the reader to \cite{MR1117631}.

When $X$ is in $\DDh(R)$, there is a complex $P$ consisting of projective modules with $P_i = 0$ for $i\ll 0$ such that $P\qis X$. Such a complex $P$ is called a projective resolution of $X$. In this case we can compute $\RHom( X, \blank)$ by setting $$\RHom( X, \blank) = \Hom_R(P,\blank)$$ Flat and injective resolutions are similarly defined. Since $R$ has enough projectives and injectives, every complex in $\DDh(R)$ admits projective (and thereby flat) resolutions and every complex in $\DDl(R)$ admits injective resolutions.

Complexes in $\DDfgb(R)$ with finite projective dimension are the \emph{perfect complexes}. The subcategory of $\DD(R)$ of complexes of finite injective dimension plays a central role in this paper and we denote it $\II(R)$.
\begin{definition}
	Let $X\in \DDh(R)$ and $Y \in \DDl(R)$. We define
	\[\pd_R(X) \coloneqq \inf \left \{ n \left | \begin{gathered}  \text{ there exists a projective resolution $P$ of } X  \\ \ 0 \to P_n \to ... \to P_i\to 0 \text{ with }  P_n \neq 0 \end{gathered} \right \} \right. \]
	\[ \id_R(Y)  \coloneqq \sup \left\{ n \left | \begin{gathered} \text{ there exists an injective resolution $I$ of } Y \\ \ 0 \to I_i \to ... \to I_n \to 0 \text{ with }  I_n \neq0  \end{gathered} \right\} \right. \]

\end{definition}
 
\begin{remark}\label{accounting principles} Let $(R, \m ,k)$ be a local ring, $X\in \DDh(R)$. Following \cite{MR1799866}*{A.5.7, A.7.9} we have what Foxby called ``accounting principles''.
\begin{enumerate}
 \item When $V$ is an R-complex such that $\m V = 0$ then,
\begin{align*}
\sup H(V\dtensor_R X) & = \sup H(V) + \sup H(k\dtensor_R X) \\
\inf H(V\dtensor_R X) & = \inf H(V) + \inf H(k\dtensor_R X) 
\end{align*}
\item  If $X \in \DDfgb(R)$ and $Y\in \DDfgl(R)$ then the following hold
\begin{align*}
\pd_R(X) &= -\inf H(\RHom_R(X,k))= \sup H(k \dtensor_RX)\\
\id_R(Y) &= -\inf H(\RHom_R(k,Y))
\end{align*}
\end{enumerate}
\end{remark}

\begin{definition}
 A local ring $R$ is \emph{Gorenstein} if $\id_{R}(R)$ is finite.
\end{definition}

The following theorem \cite{MR1799866}*{3.3.4}, known as the Foxby equivalence, asserts that in a Gorenstein local ring, the categories of finite projective dimension and finite injective dimension coincide.
 \begin{theorem} \label{Foxby2}  Let $(R,\m,k)$ be a local Gorenstein ring, $X\in \DDb(R)$. Then
$$\pd_R(X) < \infty \iff \fd_R(X) < \infty \iff \id_R(X) < \infty$$
\end{theorem} 

Foxby \cite{MR0447269}*{2.1} also proved the converse to Theorem \ref{Foxby2}. Although in the original paper it was stated for modules, it is well known to be true for complexes as well. For convenience, we give a self contained proof using the terminology and properties established above.
\begin{theorem}\label{foxby1}
Let $(R,\m,k)$ be a local ring. If there exists a complex $X\in \DDfgb(R)$ with $H(X) \neq 0$ such that both $\pd_R(X)$ and $\id_R (X) $ are finite, then $R$ is Gorenstein.
\end{theorem}
\begin{proof}
 We have quasi-isomorphisms  $$\RHom_R(k,X) \qis \RHom_R(k,R\dtensor_R X) \qis \RHom_R(k,R)\dtensor_R X$$
The first quasi-isomorphism is trivial and the second follows from \cite{MR0447269}*{1.1.4} since $X$ is perfect. Also, since $X \in \DDfgb(R)$ we have $\inf H(k\dtensor_R X)$ is finite. As $H(X)\neq 0$, we get from \ref{accounting principles}:
\begin{align*}
\id_R( X) & = -\inf H(\RHom_R(k,X))\\
&=-\inf H(\RHom_R(k,R )\dtensor_RX)\\
& = -\inf H(\RHom_R(k,R)) - \inf H(k\dtensor_R X)\\ 
& =  \id_R R - \inf H(k\dtensor_R X) \end{align*} Thus $\id_R (R) < \infty \iff \id_R (X)<\infty$.\end{proof}

\subsection*{Thick subcategories and generation}
Thick subcategories play a critical role in the proofs in Section \ref{mainsection}. Here we recall the definition and give some examples, following the formulation given in \cite{MR2681707}*{\textsection 1}
\begin{definition}
A non-empty subcategory $\mathcal{T}$ of $\DD(R)$ is \emph{thick} if it is additive, closed under taking direct summands and for every exact triangle $$X \to Y \to Z \to \Sigma X$$ if any two of $X, Y,Z$ belong to $\mathcal{T}$, so does the third. From the definition it is clear that intersections of thick subcategories is again thick.
\end{definition}
\begin{example}\label{ex_thick} The subcategories $\DDfg(R), \ \DDfl(R),$  $\II(R)$ and the subcategory of perfect complexes are all thick in $\DD(R)$, see for example \cite{MR2225632}*{3.2}. It follows immediately that $\IIfg(R)$ and $\IIfl(R)$ are thick as well.
 
\end{example}

\begin{definition}
The \emph{thick subcategory generated by $X\in\DD(R)$}, denoted $\Thick_R(X)$, is the smallest thick subcategory that contains $X$.  It is the intersection of all thick subcategories of $\DD(R)$ containing $X$.
\end{definition}
\begin{example}\label{thickEx}
	We always have $\Thick_R(R)$ are the perfect complexes. When $(R, \m,k)$ is local we have $\Thick_R(k)= \DDflb(R)$.
\end{example}

For any $X\in \DD(R)$ one can construct $\Thick_R (X)$ as follows: Set $\Thick_R^0(X) = \{0\}$. The objects of $\Thick_R^1(X)$ are direct summands of finite direct sums of shifts of $X$. For each $n \ge 2$, the objects of $\Thick_R^n(X)$ are direct summands of objects $U$ such that $U$ appears in an exact triangle $$U' \to U \to U'' \to \Sigma U'$$ where $U' \in \Thick_R^{n-1}(X)$ and $U''\in \Thick_R^1(X)$. The subcategory $\Thick_R^n(X)$ is the \emph{$n$th thickening} of $X$. Every thickening embeds in the next one thus we have a filtration:
$$\{0\} = \Thick_R^0(X) \subseteq \Thick_R^1(X) \subseteq \Thick_R^2(X) \subseteq ... \subseteq \bigcup_{n \ge 0} \Thick_R^n(X)$$

It is clear that $\bigcup_{n \ge 0} \Thick_R^n(X)$ is a thick subcategory. By construction it is the smallest thick subcategory containing $X$ hence $$\Thick_R(X) = \bigcup_{n \ge 0} \Thick_R^n(X)$$ For a broader discussion see, for example, \cite{MR2681707}*{\S 1}. This discussion motivates the following terminology: An $R$-complex in $\Thick_R(X)$ is \emph{finitely built} from $X$.
\begin{definition}
 The \emph{support} of an $R$-complex $X$ is $$\Supp_R(X)\coloneqq \{\p \in \spec(R) \mid H(X)_{\p} \neq 0 \}$$
\end{definition}
When $X\in \DDfgb(R)$ the support is $$\Supp_R(X) = V(\ann_R(H(X)).$$
If $N\in \Thick_R(M)$, then from the construction it follows that $\Supp_R(N) \subseteq \Supp_R(M)$. Indeed, since localization is an exact functor, if $H(M)_{\p}=0$ for some $\p \in \spec(R)$ then inductively $H(N)_{\p}=0$ for every $N$ in $ \Thick_R^i(M)$ for all $i$. 

Hopkins \cite{MR932260}*{11} and Neeman \cite{MR1174255}*{1.2, 2.8} proved the following result which asserts that the converse is true when both $M$ and $N$ are perfect complexes.
\begin{theorem}\label{hopkins neeman} Let $R$ be a commutative Noetherian ring. Given perfect $R$-complexes $N$ and $M$, if $\Supp_R N \subseteq \Supp_R M$ then $N$ is finitely built from $M$. \hfill{$\qed$}
\end{theorem}

\subsection*{Loewy Length}
Another important element in this work is the Koszul complex. We recall the definition of Koszul complexes and Loewy length.
\begin{definition}\label{koszuldef}
 The \emph{Koszul complex} on $x\in R$ is the $R$-complex $$K(x) \coloneqq 0 \to R \xrightarrow{x} R \to 0$$ concentrated in degrees $0$ and $1$.
Given a sequence $\mathbf{x} =(x_1, ... ,x_n)$ the Koszul complex on $\mathbf{x}$ is $$K(\mathbf{x}) \coloneqq K(x_1) \tensor_R K(x_2) \tensor_R ... \tensor_R K(x_n)$$ with the convention that $K(\varnothing) = R$.
\end{definition}
Set $K^R$ to be the Koszul complex on a minimal generating set of $\m$. Since $K^R$ is a perfect complex, we have that $K^R \in \Thick_R(R)$. It follows, that $K^R\dtensor_RX $ is in $ \Thick_R (X)$ for every $X\in \DD(R)$.
\begin{definition}\label{def_LL}
Let $(R, \m, k)$ be a local ring, $X$ an $R$-complex. The \emph{Loewy length} of $X$ is defined to be
$$\lowlen_R(X) \coloneqq \inf\{ i\in \NN\mid \m^i \cdot X =0\}$$ Following \cite{MR2197067}*{6.2}, the \emph{homotopical Loewy length} of $X$ is defined to be
$$\lowlen_{\DD(R)}(X) \coloneqq \inf\{\lowlen_R(V) \mid V\simeq X\}$$
\end{definition}
Avramov, Iyengar and Miller \cite{MR2197067}*{6.2} prove that the homotopical Loewy length satisfies the following finiteness property. 
\begin{theorem}
Let $(R, \m, k)$ be a local ring, and $K^R$ be the Koszul complex on a minimal generating set of $\m$. For any complex X we have 
\[ \pushQED{\qed} \lowlen_{\DD(R)} (K^R \dtensor_R X) \le \lowlen_{\DD(R)} K^R < \infty  \qedhere \popQED \]
\end{theorem} 
We say that a homomorphism $\varphi \colon (R,\m,k) \to (S,\n,l)$ is a \emph{deep local homomorphism} if $\varphi(\m) \subseteq \n^{c}$ where $c=\lowlen_{\DD(S)} K^S$.
The following corollary is often used in the literature to prove various results for deep local homomorphism.
\begin{lemma}\label{finiteLL} If $\varphi \colon (R,\m,k) \to (S,\n,l)$ is a deep local homomorphism, then in $\DD(R)$ the complex $K^S$ is quasi-isomorphic to $H(K^S)$.
 \end{lemma}
\begin{proof}
By \ref{def_LL} there exist a complex $V$ such that $K^S \qis V$ and $\n^cV =0$. As $\varphi(\m) \subseteq \n^c$, this yields that $\m V = 0$. Hence the $R$ action on $V$ factors through the map $R \to R/\m =k$. Since $k$ is a field, for every $V \in \DD(k)$ we have $V \qis H(V)$. In particular, $K^S \qis H(K^S)$ in $\DD(k)$ so the same is true in $\DD(R)$.
\end{proof}

\section{Homological Dimension and the Derived Base Change}\label{mainsection}
Let $\varphi \colon R \to S$ be a homomorphism. There is a naturally defined functor $F^\varphi$ from the category of $R$-complexes to the category of $S$-complexes by setting $$F^\varphi(\blank) \coloneqq S \tensor_R \blank$$ 
We write $$\LF \colon \DD(R) \to \DD(S) \text{ by } \LF(\underline{\hspace{8pt}})= S\dtensor_R \underline{\hspace{8pt}}$$ for the induced functor on $\DD(R)$. 
\begin{remark} \label{non zero}
Let $\varphi \colon (R,\m,k) \to (S,\n,l)$ be a local homomorphism. \begin{enumerate} 
 \item For every perfect complex $X$ the complex $\LF(X)$ is perfect in $\DD(S)$. Indeed, as $X$ is perfect, there exists a finite free resolution $F\qis X$. Then $\LF(X) \qis S\tensor_R F$ which is a finite complex of free $S$ modules.
\item For every $X\in \DDfgh(R)$ such that $H(X) \neq 0$ we have $H(\LF(X)) \neq 0$. Indeed, we may assume $H_0(X) \neq 0$ and $H_i(X)=0$ for all $i<0$. We have $$H_0(S \dtensor_R X) \cong S \tensor_R H_0(X)$$ Applying $S \tensor_R \blank$ to the surjection $H_0(X) \surject H_0(X)/\m H_0(X)\surject k$ we get $H_0(S \dtensor_R X) \surject S \tensor_R k \cong \dfrac{S}{\m S} \neq 0$ as $\varphi$ is a local homomorphism.
\end{enumerate}
\end{remark}

\begin{proposition}\label{finite pd}
 Let $\varphi \colon (R,\m,k) \to (S,\n,l)$ be a deep local homomorphism. For any $X\in \DDfgh(R)$ the complex $\LF(X)$ is homologically bounded above if and only if $X$ has finite projective dimension in $\DD(R)$.
\end{proposition}
\begin{proof}
The if part is clear. For the converse, by the disscusion after \ref{koszuldef} the complex $K^S \dtensor_R X \qis K^S \dtensor_S (S\dtensor_R X) $ is in $ \Thick_S(S \dtensor_R X)$. Example \ref{ex_thick} yields $$\sup H(S \dtensor_R X) < \infty \implies \sup H( K^S \dtensor_R X)  < \infty.$$
By Corollary \ref{finiteLL}, the complex $K^S \simeq H(K^S)$ in $\DD(R)$ and $H(K^S)$ is a $k$-vector space as an $R$-complex, one gets by the Künneth formula 
\begin{align*} H(K^S \dtensor_R X ) & \cong H( H(K^S) \dtensor_R X )\\ & \cong H( H(K^S) \tensor_k (k\dtensor_R X ))\\ & \cong H(K^S) \tensor_k H(k \dtensor_R X)\end{align*}
Since $H( K^S \dtensor_R X)$ is bounded, so is $H(k \dtensor_R X)$. Therefore $\pd_R(X) < \infty$ by Remark \ref{accounting principles} (ii).
\end{proof}
\begin{corollary}\label{inj-DLH}
 Let $\varphi \colon (R,\m,k) \to (S,\n,l)$ be a deep local homomorphism. If there exists an $X \in \IIfg(R)$ with $H(X) \neq 0$ and $\id_S( \LF(X))< \infty$ then $R$ is Gorenstein.
\end{corollary}
\begin{proof}
 By Remark \ref{non zero}(b), the homology $H(\LF(X)) \neq 0.$  By hypothesis $X$ is in $ \DDfgb(R)$. Hence $$\id_S( \LF(X))< \infty \implies \sup H(\LF(X)) < \infty$$ Therefore $\pd_R(X) < \infty$ by Proposition \ref{finite pd}. Theorem \ref{foxby1} now shows that $R$ is Gorenstein.
\end{proof}

\begin{remark}
 In the context of the Corollary \ref{inj-DLH}, if there exists an $X\in \DDfgb(R)$ with $H(X) \neq 0$ such that $\id_R ( \LF(X)) < \infty$ then $R$ is regular. Indeed, if $\id_R( \LF(X) ) < \infty$ then following the lines of the proof of Proposition \ref{finite pd}, we see that $\id_R (K^S \dtensor_R X )< \infty$. It follows that $\id_{R} (k \dtensor_R X)$ is finite and therefore $\id_R(k)<\infty$ which implies that $R$ is regular. This gives another proof of a result of Avramov, Hochster, Iyengar and Yao \cite{MR2915536}*{5.3}.
\end{remark}
Our main result concerns the finiteness of injective dimension with respect to the derived base change over contracting endomorphisms. 
\begin{definition}
 Let $(R,\m,k)$ be a local ring. An endomorphism $\varphi\colon R \to R$ is said to be \emph{contracting} if $\varphi^i(\m) \subseteq \m^2$ for some $i\ge 1$.
\end{definition}
\begin{remark} If $\varphi$ is a contracting endomorphism then $\varphi^i$ will be a deep local homomorphism for each $i \gg0$.
\end{remark}
If $\varphi$ is an endomorphism on $R$, then we define $\phiR$ to be $R$ with the right module structure induced by $\varphi$. Proposition \ref{finite pd} shows that given a contracting endomorphism $\varphi$ and a complex $X$ then for large enough $i$ the complex $\LF^i(X)$ is bounded if and only if $X$ has finite projective dimension. However, there are examples of complexes of infinite projective dimension for which $\LF(X)$ is homologically bounded. For example, let 
$$R=\dfrac{k[ x,y ]}{(x^3,y^3)}$$
Set $\varphi(x)=y$ and $\varphi(y)=y^2$. One can check that $\LF(x)\qis (y)$ but $(x)$ has infinite projective dimension. A natural question to ask is when does $\sup H(\LF(X))<\infty$ imply that $\sup H(\LF^i(X)) <\infty$ for all $i>0$? Our goal is to show that if $X$ has finite injective dimension then $\LF(X)$ is homologically bounded if and only if $\LF^i(X)$ is homologically bounded for every $i>0$.
\begin{definition}
 Let $(R,\m,k)$ be a local ring and $E$ the injective hull of the $R$-module $k$. For an $R$-complex M set $$M^\vee \coloneqq \Hom_R(M,E).$$
\end{definition}
We will need the following lemma; we give a proof for completeness.

\begin{lemma}\label{std lemma} Let $(R,\m,k)$ be a local ring.
	\begin{enumerate}
		\item The natural map $X \to X^{\vee\vee}$ is a quasi-isomorphism for all $X\in \DDflb(R)$.
		\item The complex $X^{\vee}$ is perfect with finite length homology for all $X\in \IIfl(R)$.
	\end{enumerate}
\end{lemma}
\begin{proof}

For (i), we observe that $\{X\in \DD(R) \mid X\qis X^{\vee \vee} \}$ form a thick subcategory. When $X\in \DDflb(R)$ one can show that by induction on the total length of $H(X)$ that $X \in \Thick_R(k)$. Clearly $k \qis k^{\vee \vee}$ so it follows that $X \qis X^{\vee \vee}$ for all $X \in \DDflb(R)$.

For (ii), take an injective resolution $I$ of $X$. Since $\Supp(X)=\{\m\}$, the injective resolution $I$ is a finite complex where all the modules are direct sums of $E$. Hence by Matlis duality, $X^{\vee}$ is quasi-isomorphic to a bounded complex of free $\widehat{R}$ modules. As $\widehat{R}$ is flat, $$\sup(k \dtensor_R X^\vee) = \sup(k \dtensor_RX^\vee \dtensor_R \widehat{R})$$ Since $X^{\vee}$ is perfect in $\DD(\widehat{R})$ it is also perfect in $\DD(R)$.	
\end{proof}

\begin{proposition}\label{HN dual}
For all $X\in \IIfl(R)$ with $H(X)\neq 0$, one has $\Thick_R (X)= \IIfl(R)$. 
\end{proposition}
\begin{proof}
 By Example \ref{ex_thick}, $\Thick_R(X) \subseteq \IIfl(R)$, so it suffices to show that for all $Y \in \IIfl(R)$ we have $Y$ is in $\Thick_R(X)$. Let $Y \in \IIfl(R)$ with $H(Y) \neq0$. By Lemma \ref{std lemma}(ii) $X^{\vee}$ and $Y^{\vee}$ are complexes with finite length homology and finite projective dimension over $R$. In particular, $X^\vee$ and $Y^\vee$ are both perfect, and $$\Supp (X^\vee) = \{ \m \}= \Supp (Y^\vee)$$ Theorem \ref{hopkins neeman} yields that 
$$\Thick_R (X^\vee) = \Thick_R (Y^\vee)$$
Applying Matlis duality again, and noting that $X^{\vee \vee} \qis X$ by Lemma \ref{std lemma}(i), we see that $\Thick_R (X) = \Thick_R (Y).$ In particular, $Y $ is in $ \Thick_R(X)$.
\end{proof}
\begin{lemma}\label{iterate phi}
 Let $\varphi \colon (R,\m,k) \to (R,\m,k)$ be contracting endomorphism. If for some $X \in \IIfl(R)$ with $H(X) \neq0$ the injective dimension of $\LF(X)$ is finite then the injective dimension of $\mathsf{L}F^{\varphi^i} (Y)$ is finite for all $i \ge 1$ and all $Y \in \IIfl(R)$.
\end{lemma}
\begin{proof}
 By Remark \ref{non zero} $H(\LF(X)) \neq 0$ when $H(X) \neq 0$. Proposition \ref{HN dual} shows that $\Thick_{ \DD (R) } (X)= \IIfl(R)$. Since $\LF(\blank)$ is an exact functor it follows that $\LF(Y)$ is in $\Thick_R (\LF(X))$ for every $Y\in \IIfl(R)$. By hypothesis $\LF(X) \in \IIfl(R)$, hence the functor $\LF(\blank)$ takes $\IIfl(R)$ to $\IIfl(R)$, but this implies that $  \mathsf{L}F^{\varphi^2} (Y) \cong \LF(\LF(Y))$ has finite injective dimension for every $Y \in \IIfl(R)$. By induction on $i$, we have $\mathsf{L}F^{\varphi^i} (Y)$ is finite for all $i \ge 1$ and all $Y \in \IIfl(R)$.
\end{proof}
The following theorem is a restatement of Theorem \ref{main result}.
\begin{Theorem}
Let $\varphi \colon (R,\m,k) \to (R,\m,k)$ be a contracting endomorphism. The following are equivalent.

\begin{enumerate}
 \item $R$ is Gorenstein.
\item There exists an $X \in \IIfg(R)$ with $H(X) \neq 0$ and $\LF(X) \in \IIfg(R)$.
  \item For every $X \in \IIfg(R)$ we have $\LF(X) \in \IIfg(R)$.
\end{enumerate}
\end{Theorem}
\begin{proof}
 (i) $\implies$ (iii). Theorem \ref{Foxby2} shows that $\IIfg(R)$ are the perfect complexes. So for every $X\in \IIfg(R)$ the base change $\LF(X)$ is also perfect and hence in $\IIfg(R)$.
 
(iii) $\implies$ (ii). We need to show that for every local ring there exists a complex $X\in \IIfg(R)$ with $H(X) \neq0$. Let $E$ be the injective hull of the residue field, $K^R$ the Koszul complex of $R$. The complex $K^R \tensor_R E$ is Artinian and $\m H( K^R \tensor_R E) =0$, hence it has finite length homology.

 (ii) $\implies$ (i), Let $K^R$ be the Koszul complex of $R$. Let $X\in \IIfg(R)$ with $\id_{R} (\LF( X)) < \infty$. Since $K^R \dtensor_R X$ still has finite injective dimension and  non-zero finite length homology, Lemma \ref{iterate phi} shows that 
  $$\id_{R} (\LF^i(K^R \dtensor_R X)) < \infty \quad\text{for all } i \ge 1$$ 
  Setting $c= \lowlen_{\DD(R)} K^R$ we can take $i$ large enough so that $\varphi^i(\m) \subseteq \m^c$. Corollary \ref{inj-DLH} yields that $R$ is Gorenstein.
\end{proof}

Theorem \ref{main result} is the derived analogue of the following result by Falahola and Marley. \cite{FM}*{Theorem 3.1}
\begin{theorem} \label{FM thm}
 Let $(R,\m,k)$ be a Cohen-Macaulay local ring, $\varphi$ be a contracting endomorphism. Suppose that $\omega_R$ is a canonical module for $R$, then $\id_{\phiR} F^\varphi(\omega_R)$ in finite if and only if $R$ is Gorenstein.
\end{theorem}

\begin{remark}
 Falahola and Marley \cite{FM}*{Example 3.8} show that Theorem \ref{FM thm} fails if we replace $\omega_R$ with a general dualizing complex $C$. They ask \cite{FM}*{Question 3.9}, if $R$ has a dualizing complex $C$ is it true that $\id_{R} ({}^{\varphi}R \dtensor_R C) < \infty$ if and only if $R$ is Gorenstein? Since a dualizing complex is in $\IIfg(R)$, Theorem \ref{main result} shows in particular that if $C$ is a dualizing complex in $\DD(R)$ then $\id_R(\LF(C))$ being finite implies that $R$ is Gorenstein, giving an affirmative answer.  
 
\end{remark}

\bibliographystyle{amsalpha}

\end{document}